\documentclass[10pt]{article}
\usepackage{amsmath,amssymb,wasysym,amsfonts,amscd,amsthm}
\usepackage{subfig,wrapfig}
\usepackage{graphicx}
\usepackage{epstopdf}
\usepackage{fullpage}
\newtheorem{theorem}{Theorem}[section]
\newtheorem{lemma}[theorem]{Lemma}

\newtheorem{definition}[theorem]{Definition}
\newtheorem{question}[theorem]{Question}
\newtheorem{corollary}[theorem]{Corollary}

\newtheorem{claim}[theorem]{Claim}
\newtheorem{problem}[theorem]{Problem}

\newcommand{\size}{\operatorname{size}}

\title{Sum list coloring, the sum choice number, and sc-greedy graphs}
\author{Michelle Lastrina \and Michael Young}
\date{May 6, 2013}

\begin{document}
\maketitle

\begin{abstract}
Let $G=(V,E)$ be a graph and let $f$ be a function that assigns list sizes to the vertices of $G$.  It is said that $G$ is $f$-choosable if for every assignment of lists of colors to the vertices of $G$ for which the list sizes agree with $f$, there exists a proper coloring of $G$ from the lists.  The sum choice number is the minimum of the sum of list sizes for $f$ over all choosable functions $f$ for $G$.  The sum choice number of a graph is always at most the sum $|V|+|E|$.  When the sum choice number of $G$ is equal to this upper bound, $G$ is said to be sc-greedy.  In this paper, we determine the sum choice number of all graphs on five vertices, show that trees of cycles are sc-greedy, and present some new general results about sum list coloring.
\end{abstract}


\section{Introduction}
Let $G=(V,E)$ be a graph on $n$ vertices and $f:V\rightarrow\mathbb{N}$ be a \textit{size function} that assigns to each vertex of $G$ a list size.  Let an \textit{$f$-assignment} $L:V\rightarrow2^{\mathbb{N}}$ be an assignment of lists of colors to the vertices of $G$ such that $|L(v)|=f(v)$ for all $v\in V$.  The graph $G$ is said to be \textit{$f$-choosable} if every $f$-assignment is $L$-colorable.  A choosable size function is called a \textit{choice function}.  For a choice function $f$, define $\size(f)=\sum_{v\in V}f(v)$.  The \textit{sum choice number} $\chi_{SC}(G)$ is the minimum of $\size(f)$ over all choice functions for $G$.
It is not hard to see that for any graph $G$, $\chi_{SC}(G) \le |V(G)| + |E(G)|$.  This upper bound is provided by a greedy coloring, see Lemma \ref{GB} for a proof of this result.  When equality holds in the previous inequality,
$G$ is said to be \textit{sc-greedy}.  The sum $|V(G)| + |E(G)|$ is called the \textit{greedy bound} and denoted by $GB(G)$, or $GB$ when $G$ is implied.
\begin{lemma}\label{GB}
For any graph $G$, $\chi_{SC}(G) \le |V(G)| + |E(G)|.$
\end{lemma}
\begin{proof}
Let $v_1,\ldots,v_n$ be an ordering of the vertices.  Let $f(v_i) = 1 + |\{v_j:j<i\text{ and } v_iv_j\in E(G)\}|$.  Using this ordering, a greedy coloring from arbitrary lists of the prescribed sizes will give a proper coloring for any such list assignment.
\end{proof}

Observe that list coloring and the list chromatic number, or choice number, $\chi_l(G)$ are related to sum list coloring and the sum choice number.  In sum list coloring, the list sizes are allowed to vary and one seeks to minimize the sum of list sizes over all vertices.  With list coloring, one tries to minimize the sizes of lists $L$ of colors assigned to the vertices of a graph such that the graph is $L$-colorable.  The choice number $\chi_l(G)$ is the minimum integer $k$ for which a graph $G$ is $L$-colorable given $|L(v)| = k$ for all $v\in V(G)$.  It is not hard to see that $\chi_{SC}(G)/n \le \chi_l(G)$.  Moreover, for some graphs $G$, it is the case that $\chi_{SC}(G)/n$ is significantly smaller than $\chi_l(G)$.  In \cite{FK09}, F\"{u}redi and Kantor showed that there is an infinite family of complete bipartite graphs such that $\chi_{SC}(G)/n$ can be bounded while the minimum degree increases without bound.  This is not the case for the list chromatic number $\chi_l(G)$.  Specifically, it is observed that a choosable function can be found for complete bipartite graphs such that the average list size does not necessarily grow with the average degree.

To show that $\chi_{SC}(G) = m$, one must provide a choice function $f$ of size $m$ for $G$ and show that for each size function $g$ of size $m-1$, there is a $g$-assignment $L$ for which $G$ is not $L$-colorable.  In other words, if $\size(g)=m-1$, then $g$ is not a choice function for $G$.
This paper is devoted to showing that certain graphs made up of cycles are sc-greedy and computing the sum choice number of all graphs on five vertices.  Additionally, we provide examples of certain graphs that illustrate some general results and determine information about the sum choice number of other graphs.

Sum list coloring was introduced by Isaak \cite{Isaak02,Isaak04} in 2002.  It is a fairly new topic in graph theory, so there is much to be discovered.  For more on sum list coloring see also \cite{Heinold06,Heinold07,Heinold09,BBBD06,FK09}.  In particular, \cite{Heinold07} is a survey of all sum list coloring results up to 2007.  We note here that some of the results in this paper can also be found in \cite{Lastrina12}.  In the next section, we will further discuss Heinold's survey and present some known results about sum list coloring.  First, we define some of the graphs that will be used throughout this paper.

\subsection{Definitions of graphs}
Let $P_n$ denote the path on $n$ vertices and $C_n$ denote the cycle on $n$ vertices.
The graph $W_k$ is the \textit{wheel} on $k+1$ vertices $v_0,v_1,\ldots,v_k$ formed by taking the cycle $v_1v_2\ldots v_kv_1$ and adding the vertex $v_0$ so that it is adjacent to each of the vertices in the $k$-cycle.  The graph $BW_k$ is the \textit{broken wheel} on $k+1$ vertices $v_0,v_1,\ldots,v_k$ which is isomorphic to the graph $W_k - v_1v_k$, the wheel on $k + 1$ vertices without an edge between $v_1$ and $v_k$.  Note that this graph is often referred to as the fan graph $F_k$ on $k + 1$ vertices.  See Figure \ref{fig:BWandW} for examples of $W_k$ and $BW_k$.
\begin{figure}[h]
\begin{center}
\subfloat[][$W_8$]{\label{fig:W8}\includegraphics{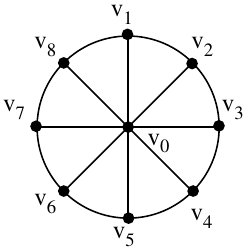}}
\qquad
\subfloat[][$BW_4$]{\label{fig:BW4}\includegraphics{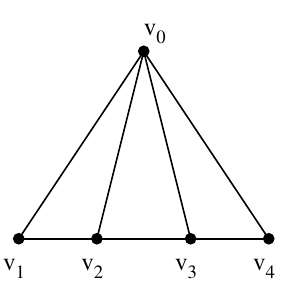}}
\end{center}
\caption{Examples of a wheel and broken wheel.}
\label{fig:BWandW}
\end{figure}

\begin{definition}\label{def:prod}
The \textbf{Cartesian product} $G\times H$ is the graph with vertex set $V(G)\times V(H)$ where any two vertices $(u,u')$ and $(v,v')$ are adjacent in $G\times H$ if and only if either (1) $u=v$ and $u'$ is adjacent to $v'$ in $H$, or (2) $u'= v'$ and $u$ is adjacent to $v$ in $G$.
\end{definition}

\begin{definition}
The \textbf{theta graph} $\Theta_{k_1,k_2,k_3}$ is the union of three internally disjoint paths with $k_1,k_2,k_3$ internal vertices, respectively, that have the same two distinct end vertices.
\end{definition}

\begin{definition}
The graph $G^k$, the \textbf{$k$th power} of a graph $G$, is the graph with the same vertex set as $G$ and an edge between vertices $u$ and $v$ if and only if there is a path of length at most $k$ between $u$ and $v$ in $G$.
\end{definition}

We define the following graph that is obtained by laying cycles of arbitrary and varying lengths greater than three end to end so that they share an edge.
\begin{definition}
A graph $G$ is called a \textbf{path of $k$ cycles}, or \textbf{path of cycles}, if $G=\bigcup_{i=1}^k G_i$ where 
\begin{enumerate}
\item each $G_i$ is a cycle of length $a_i\ge4$ for $i=1,\ldots,k$,
\item $V(G_{i-1})\cap V(G_i)=\{t_i,b_i\}$ for all $i=2,\ldots,k$,
\item $E(G_{i-1})\cap E(G_i)=\{t_ib_i\}$ for all $i=2,\ldots,k$,
\item $t_i,b_i\not\in V(G_j)$ for all $j\ne i-1,i$
\item If $w_i\in V(G_i)-\{t_i,b_i,t_{i+1},b_{i+1}\}$, then $w_i\not\in V(G_j)$ for $j\ne i$.
\item $G$ can be drawn in the plane so that the weak dual of $G$ is a path.
\end{enumerate}
\end{definition}
See Figure \ref{fig:cyclepath} for an example of a path of cycles.
\begin{figure}[h]
\begin{center}
\includegraphics{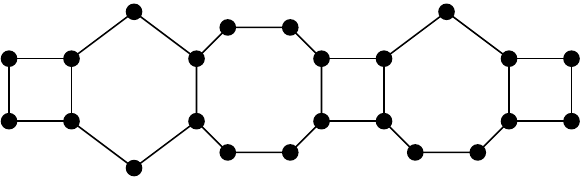}
\caption{An example of a path of cycles.}
\label{fig:cyclepath}
\end{center}
\end{figure}

We also define a graph that is obtained by laying cycles of arbitrary and varying lengths greater than three along a special tree-like structure so that they share an edge.
\begin{definition}
A graph $G$ is called a \textbf{tree of $k$ cycles}, or \textbf{tree of cycles}, if $G=\bigcup_{i=1}^k G_i$ where 
\begin{enumerate}
\item each $G_i$ is a cycle of length $a_i\ge4$ for $i=1,\ldots,k$
\item for all pairs $i,j$, it must be that $V(G_i)\cap V(G_j)=\emptyset$ or $V(G_i)\cap V(G_j)=\{u,v\}$ for two adjacent vertices $u,v\in V(G)$
\item If $V(G_i)\cap V(G_j)=\{u,v\}$, then $u,v\not\in V(G_l)$ for all $l\ne i,j$.
\item If $w_i\in V(G_i)$, then $w_i\not\in V(G_j)$ for all $j\ne i$ unless $V(G_i)\cap V(G_j)=\{u,v\}$ and $w_i\in\{u,v\}$.
\item $G$ can be drawn in the plane so that the weak dual of $G$ is a tree.
\item $G$ can be drawn in the plane so that in the dual of $G$, the vertex corresponding to the unbounded face of $G$ is adjacent to all other vertices.
\end{enumerate}
\end{definition}
See Figure \ref{fig:cycletree} for an example of a tree of cycles.
\begin{figure}[h]
\begin{center}
\includegraphics{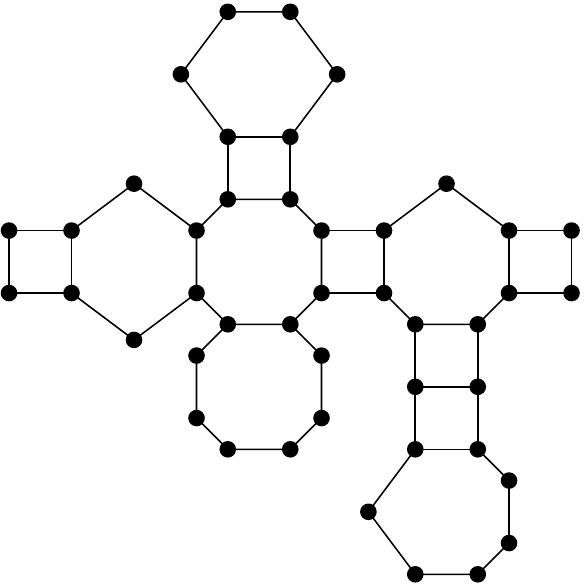}
\caption{An example of a tree of cycles.}
\label{fig:cycletree}
\end{center}
\end{figure}
Note that a path of cycles is a special case of a tree of cycles that occurs when the underlying tree-like structure is a path.

\section{Background and preliminaries}
A survey by Heinold \cite{Heinold07}, and a more recent unpublished version from 2011, compiles results about sum list coloring, the sum choice number, and states open problems in the area.  Specifically, it contains a summary of graphs that are known to be sc-greedy, as well as other graphs whose sum choice number is known.
These graphs, as well as previous results included in the survey that will be used in this paper, are stated here for convenience.
Note that the number of graphs known to be sc-greedy is not large, nor is the number of graphs whose sum choice number is known.

The following graphs are known to be sc-greedy: complete graphs \cite{Isaak04}, paths, cycles, trees, and cycles with pendant paths.  Also, the Petersen graph is sc-greedy, as well as $P_2\times P_n$ and the theta graph $\Theta_{k_1,k_2,k_3}$ unless $k_1=k_2=1$ and $k_3$ is odd \cite{Heinold06}.
Paths of cycles are sc-greedy.
While this result is stated without proof in \cite{Heinold07} and elsewhere, this paper contains a proof of the more general result that certain trees of cycles are sc-greedy.
The graph $P_n^2$ is also sc-greedy \cite{Heinold06}.
Additional graphs whose sum choice number is known are $P_3\times P_n$ \cite{Heinold06}, $K_{2,n}$ \cite{BBBD06}, $K_{3,n}$ \cite{Heinold06}, $K_2\times K_n$ \cite{Isaak02}, and $K_3\times K_3$ \cite{BBBD06}.  See Table \ref{tab:sc} for more information.  In most cases, these graphs are not sc-greedy.

The graph $K_{2,3}$ is the smallest graph which is not sc-greedy.  Note that $K_{2,3}$ is a graph on five vertices and all graphs on at most four vertices are sc-greedy.  There exist other graphs on five vertices that are not sc-greedy.  These will be presented in Section \ref{sec:five}.

\renewcommand{\arraystretch}{1.25}
\begin{table}
\begin{center}
\caption{Graphs that are not generally sc-greedy whose sum choice number is known \cite{BBBD06,Isaak02,Heinold06}.}
\label{tab:sc}
\begin{tabular}{|l|c|c|l|}
  \hline
  $G$  &  $\chi_{SC}(G)$  &  $GB(G)$  &  sc-greedy? \\ \hline
  $P_3\times P_n$  &  $8n-3-\left\lfloor\frac{n}{3}\right\rfloor$  &  $8n-3$  &  $n=1,2$ \\ \hline
  $K_{2,n}$  &  $2n+1+\left\lfloor\sqrt{4n+1}\right\rfloor$  &  $3n+2$  &  $n=1,2$ \\ \hline
  $K_{3,n}$  &  $2n+1+\left\lfloor\sqrt{12n+4}\right\rfloor$  &  $4n+3$  &  $n=1$ \\ \hline
  $K_2\times K_n$  &  $n^2+\left\lceil\frac{5n}{3}\right\rceil$  &  $n^2+2n$  &  $n=1,2$ \\ \hline
  $K_3\times K_3$  &  $25$  &  $27$  &  no \\ \hline
  $\Theta_{1,1,2k+1}$  &  $4k+10$  &  $4k+11$  &  no \\
  \hline
\end{tabular}
\end{center}
\end{table}

Given these known results, there are some related graphs for which it would be natural to determine their sum choice number.  Let a \textit{generalized theta graph} be defined analogously to theta graphs except they may have more than three internally disjoint paths.  Then, in the future, we would like to answer the following questions.

\begin{question}
What is the sum choice number of a generalized theta graph?
\end{question}

\begin{question}
What is the sum choice number of $P_4\times P_4$?
\end{question}

A \textit{non-simple size function} is a function $f$ for which $2\le f(v)\le\deg(v)$ for all $v\in V(G)$.  Otherwise, $f$ is a \textit{simple size function}, i.e. $f(v)=1$ or $f(v)>\deg(v)$ for some $v\in V(G)$.  Answers to questions regarding sum choosability for simple size functions of a graph $G$ can be determined by answering sum choosability questions about $G-v$.  This section contains some known results that illustrate this idea.

What follows are some previous results that can be used to show a graph is sc-greedy.
Let $f$ be a size function, then $f^v$ is the size function assigned to $G - v$ where $f^v(w) = f(w) - 1$ if $vw\in E(G)$ and $f^v(w)=f(w)$ otherwise.  For a subgraph $H$ of $G$, let $f_H$ be the size function restricted to $H$.

\begin{lemma}[Isaak \cite{Isaak04}]\label{minus}
Let $G$ be a graph and $f$ a size function for $G$.
\begin{enumerate}
\item If $f(v)=1$ for some $v\in V(G)$, then $G$ is $f$-choosable if and only if $G-v$ is $f^v$-choosable.
\item If $f(v)>\deg(v)$ for some $v\in V(G)$, then $G$ is $f$-choosable if and only if $G-v$ is $f_{G-v}$-choosable.
\end{enumerate}
\end{lemma}

Let $\tau(G)$ and $\rho(G)$ be defined as follows:
\begin{eqnarray*}
  \tau(G)  &=&  \min\{\size(f):G\text{ is }f\text{-choosable and }2\le f(v)\le\deg(v)\,\forall v\in V(G)\}, \\
  \rho(G)  &=&  \min\{\chi_{SC}(G-v)+\deg(v)+1:v\in V(G)\}.
\end{eqnarray*}

\begin{lemma}[Heinold \cite{Heinold06,Heinold07}]\label{greedy}
For all graphs $G$, $\chi_{SC}(G)=\min\{\rho(G),\tau(G)\}$.  In particular, if $G-v$ is sc-greedy for all $v\in V(G)$, then $\chi_{SC}(G)=\min\{GB(G),\tau(G)\}$.
\end{lemma}

To show that a graph $G$ is sc-greedy, one can show that $G - v$ is sc-greedy for all $v\in V(G)$ and that there does not exist a non-simple choice function of size $GB-1$.  By recursively applying this idea, a known sc-greedy graph can eventually be obtained, and it suffices to show that at each step there does not exist a non-simple choice function of size one less than the greedy bound.

There are also results pertaining to how the sum choice number of a graph depends on the sum choice numbers of its blocks.
\begin{lemma}[Berliner et al. \cite{BBBD06}]\label{merge}
Let $G$ and $G'$ be such that $V(G)\cap V(G')=\{v_0\}$.  Then $$\chi_{SC}(G\cup G')=\chi_{SC}(G)+\chi_{SC}(G')-1.$$
\end{lemma}

This allows us to answer questions about graphs with cut-vertices.  Can we find an analogous result for graphs with cut-edges?  What if $G$ and $G'$ are both sc-greedy?  While these questions remain unanswered, Lemmas \ref{greedy} and \ref{merge} imply the following.

\begin{corollary}\label{main}
Let $G$ be a graph.
\begin{enumerate}
\item If $\tau(G)\ge GB(G)$ and $G-v$ is sc-greedy for all $v\in V$, then $G$ is sc-greedy.
\item If $G = G_1 \cup G_2$ with $V(G_1) \cap V(G_2) = \{v_0\}$ and $G_1$ and $G_2$ are both sc-greedy, then $G$ is sc-greedy.
\end{enumerate}
\end{corollary}

In fact, even more can be said.
\begin{lemma}[Heinold \cite{Heinold07}]\label{lem:block}
Let $G$ be a graph with blocks $G_1,\ldots,G_k$.  Then $$\chi_{SC}(G)=\sum_{j=1}^k\chi_{SC}(G_j)-k+1.$$  In particular, a graph whose blocks are sc-greedy, is itself sc-greedy.
\end{lemma}

We state the following special case explicitly, as it will be useful later on.
\begin{corollary}\label{cor:plus1}
If $G'$ is obtained by appending a vertex of degree one to a graph $G$, then $\chi_{SC}(G')=\chi_{SC}(G)+2$.  In particular, if $G$ is sc-greedy, then so is $G'$.
\end{corollary}

The following lemma is especially helpful in the case $r=0$ because it can be used to force a color on a specific vertex.
\begin{lemma}[Berliner et al. \cite{BBBD06}]\label{lem:vertforce}
Let $G$ be a graph and $f$ a choice function for $G$.  If $\size(f)=\chi_{SC}(G)+r$ for some $r\ge0$, then for any $v\in V(G)$ and any set $S$ of $r+1$ colors, there is an $f$-assignment $L$ for which every proper $L$-coloring of $G$ uses a color from $S$ on $v$.
\end{lemma}

\section{Graphs on five vertices}\label{sec:five}
In this section we explore graphs on at most five vertices 
and determine their sum choice number.  As noted earlier, all graphs on at most four vertices are sc-greedy.  This is not hard to check.
Since $K_{2,3}$ is a graph on five vertices that is not sc-greedy, we look at all other graphs on five vertices.  We note here that techniques similar to those used below are currently being used to determine the sum choice number of all graphs on six vertices.

Let $G$ be a connected graph on five vertices.  If $G$ has a cut-vertex, then $G$ is sc-greedy by Lemma \ref{lem:block}.  Thus, it remains to consider all graphs on five vertices that do not have a cut-vertex, of which there are nine nonisomorphic such graphs. These graphs are depicted, and named for convenience, in Figure \ref{fig:G5}.

\begin{figure}[h]
\begin{center}
\subfloat[][$G5.1$]{\label{fig:G5_1}\includegraphics{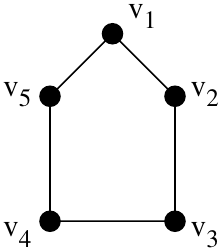}}
\quad
\subfloat[][$G5.2$]{\label{fig:G5_2}\includegraphics{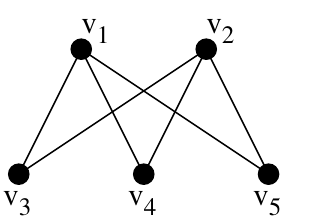}}
\quad
\subfloat[][$G5.3$]{\label{fig:G5_3}\includegraphics{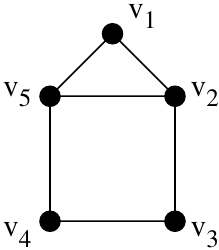}}
\quad
\subfloat[][$G5.4$]{\label{fig:G5_4}\includegraphics{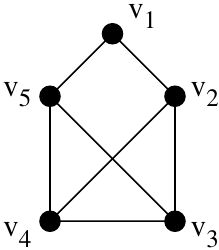}}
\quad
\subfloat[][$G5.5$]{\label{fig:G5_5}\includegraphics{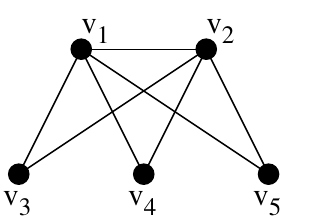}}
\qquad
\subfloat[][$G5.6$]{\label{fig:G5_6}\includegraphics{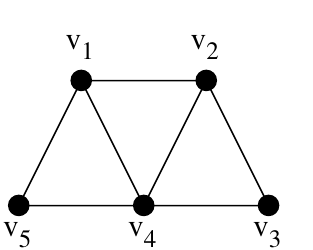}}
\qquad
\subfloat[][$G5.7$]{\label{fig:G5_7}\includegraphics{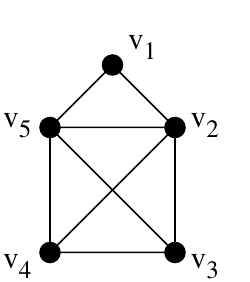}}
\qquad
\subfloat[][$G5.8$]{\label{fig:G5_8}\includegraphics{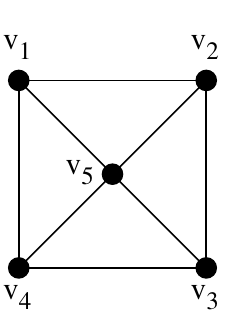}}
\qquad
\subfloat[][$G5.9$]{\label{fig:G5_9}\includegraphics{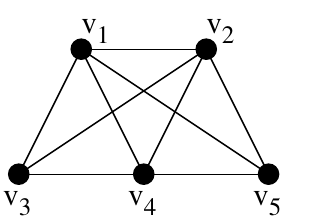}}
\end{center}
\caption{All graphs on five vertices without a cut-vertex.}
\label{fig:G5}
\end{figure}

Table \ref{tab:5vertices} summarizes the greedy bound, sum choice number, and whether or not the result was known for each of the graphs on five vertices in consideration, including if the result was determined using Sage.  After considering previous results, and ruling out additional graphs using Sage, we are left with two distinct graphs, $G5.4$ and $G5.8$, for which we must determine the sum choice number.\\

\begin{table}[h]
\begin{center}
\caption{Sum choice number of graphs on five vertices \cite{Heinold06}.}
\label{tab:5vertices}
\begin{tabular}{|c|c|c|c|c|}
  \hline
  $G$ & $\chi_{SC}(G)$ & $GB(G)$ & sc-greedy? & result \\ \hline
  $G5.1$ & 10 & 10 & yes & $C_5$ \\ \hline
  $G5.2$ & 10 & 11 & no & $K_{2,3}$ \\ \hline
  $G5.3$ & 11 & 11 & yes & $\Theta_{0,1,2}$ \\ \hline
  $G5.4$ & 11 & 12 & no &  \\ \hline
  $G5.5$ & 12 & 12 & yes & Sage \\ \hline  
  $G5.6$ & 12 & 12 & yes & $BW_4$,$P_5^2$ \\ \hline
  $G5.7$ & 13 & 13 & yes & Sage \\ \hline
  $G5.8$ & 12 & 13 & no & $W_4$ \\ \hline
  $G5.9$ & 14 & 14 & yes & Sage \\
  \hline
\end{tabular}
\end{center}
\end{table}

We note here that in the case analyses that follow, specific size functions of size $GB-1$ will be considered.  These size functions were determined using Sage and are the only such such functions that need to be considered.  This is because for all other size functions $g$ of size $GB-1$, $g$-assignments that are not list-colorable were found using Sage.  The number of distinct colors used in the list assignments found using Sage is at most two more than the maximum list size.  Also, any vertex labeling referred to in this discussion is as in Figure \ref{fig:G5}.\\ 

\noindent\textbf{The graph $G5.4$ is not sc-greedy.}  This graph has greedy bound $12$, we show that $\chi_{SC}(G5.4)=11$. It will be shown that the size function $f$ as shown in Figure \ref{fig:G5_4f} is a choice function of size $11$ for $G5.4$. Fix an arbitrary $f$-assignment $L$.  We determine information about $L$.  If $L(v_1)\cap L(v_2)=\emptyset$ or $L(v_1)\cap L(v_5)=\emptyset$, then $G5.4$ will be $L$-colorable.  This is because, without loss of generality, $G5.4-v_1v_2$ has a cut-vertex so it is sc-greedy.  So assume $L(v_1)\cap L(v_2)\neq\emptyset$ and $L(v_1)\cap L(v_5)\neq\emptyset$.\\

Case 1: There is an element $a\in L(v_1)\cap L(v_2)$ such that $a\not\in L(v_5)$.  Assign color $a$ to $v_1$, remove it from $L(v_2)$, then delete $v_1$.  Figure \ref{fig:G5_4f1} illustrates the list sizes on the remaining vertices.  This graph can be colored from lists of these sizes unless the lists are as illustrated in Figure \ref{fig:G5_4L1_}.  This implies $v_1$ should not be assigned $a$.  It also provides information about all of the lists.  There are two possible distinct list assignments obtained from this situation, see Figures \ref{fig:G5_4L1} and \ref{fig:G5_4L2}.  The graph $G5.4$ can be colored from both of these list assignments.  \\

Case 2: There is an element $a\in L(v_1)\cap L(v_2)\cap L(v_5)$.  Assign color $a$ to $v_2$ and $v_5$, remove it from $L(v_1)$, $L(v_3)$ and $L(v_4)$, then delete $v_2$ and $v_5$.  The remaining vertices can be colored in the following order: $v_1,\,v_4,\,v_3$.  This shows that $f$ is a choice function of size $11$ for $G5.4$.   \\
It is not hard to verify that a size function of size $10$ will never be a choice function for $G5.4$.  Thus, $\chi_{SC}(G5.4)=11$.\\

\begin{figure}[h]
\begin{center}
\subfloat[][$G5.4$]{\label{fig:G5_4f}\includegraphics{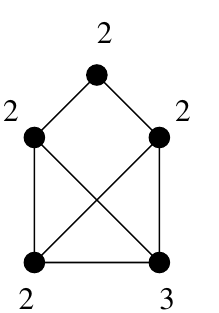}}
\quad
\subfloat[][$G5.4-v_1$]{\label{fig:G5_4f1}\includegraphics{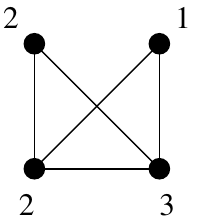}}
\quad
\subfloat[][$G5.4-v_1$]{\label{fig:G5_4L1_}\includegraphics{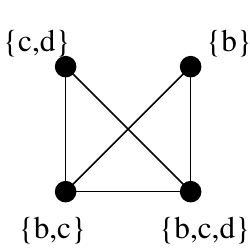}}
\quad
\subfloat[][$G5.4$]{\label{fig:G5_4L1}\includegraphics{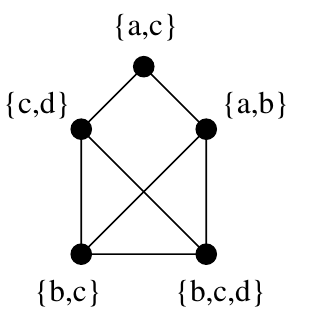}}
\quad
\subfloat[][$G5.4$]{\label{fig:G5_4L2}\includegraphics{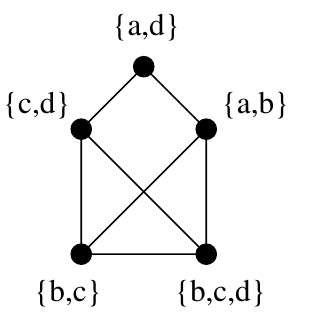}}
\end{center}
\caption{$G5.4$ is not sc-greedy.}
\label{fig:G5.4sc}
\end{figure}

\noindent\textbf{The graph $G5.8$ is not sc-greedy.}  This graph is $W_4$ and has greedy bound $13$, but we show that $\chi_{SC}(G5.8)=\chi_{SC}(W_4)=12$.  It will be shown that the size function $f$, as illustrated in Figure \ref{fig:G5_8f}, is a choice function for $G5.8$.   If there is an element $a\in L(v_2)\cap L(v_4)$, then assign $a$ to vertices $v_2$ and $v_4$.  The remaining vertices can then be colored.  Thus, assume $L(v_2)\cap L(v_4)=\emptyset$.  Without loss of generality, let $L(v_2)=\{a,b\}$, $L(v_4)=\{c,d\}$, and assume $d\in L(v_4)-L(v_5)$.  
Assign $d$ to $v_4$, then remove $d$ from the lists of adjacent vertices, and delete $v_4$ from the graph.  Figure \ref{fig:G5_8f_} illustrates the list sizes on the remaining vertices.  This graph can be colored from lists of the indicated sizes, unless the lists are as illustrated in Figure \ref{fig:G5_8L_}.  This implies $v_4$ should not be assigned $d$.  However, this indicates what all of the lists will be, and Figure \ref{fig:G5_8L} shows these lists. This graph can be colored from the provided $L$-assignment.
It is not hard to verify that a size function of size $11$ will never be a choice function for $G5.8$.  Thus, $\chi_{SC}(G5.8)=12$. \\

\begin{figure}[h]
\begin{center}
\subfloat[][$G5.8$]{\label{fig:G5_8f}\includegraphics{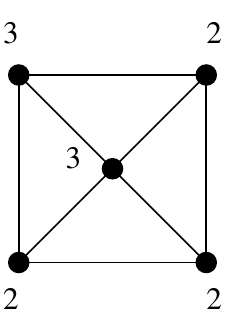}}
\quad
\subfloat[][$G5.8-v_4$]{\label{fig:G5_8f_}\includegraphics{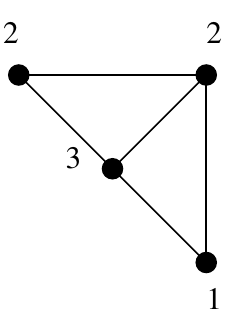}}
\quad
\subfloat[][$G5.8-v_4$]{\label{fig:G5_8L_}\includegraphics{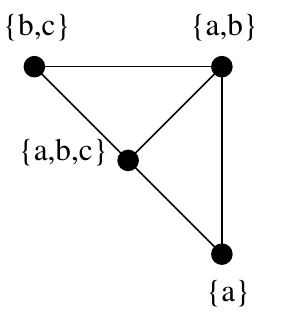}}
\quad
\subfloat[][$G5.8$]{\label{fig:G5_8L}\includegraphics{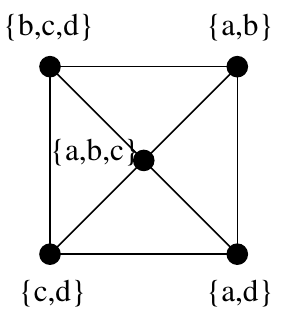}}
\end{center}
\caption{$G5.8$ is not sc-greedy.}
\label{fig:G5.8sc}
\end{figure}

\section{General results and examples}
In this section, we present new general results about sum list coloring and provide corresponding examples for illustration.
\begin{claim}\label{sub}
Let $G=(V,E)$ be a graph.  If there is a $v\in V(G)$ for which $G-v$ is not sc-greedy, then $G$ is not sc-greedy.
\end{claim}

\begin{proof}
Assume $|V|=n$ and $|E|=e$.  Let $f'$ be a choice function of minimum size for the graph $G-v$ which is not sc-greedy.  Then $\size(f')<(n-1)+(e-\deg(v))$.  Now let $f$ be a size function for $G$ such that $f(u)=f'(u)$ for all $u\neq v$ and $f(v)=\deg(v)+1$.  Then $f$ is a choice function for $G$ such that $\size(f)=\size(f')+1+\deg(v)< n+e$ and it follows that $G$ is not sc-greedy. 
\end{proof}

This claim generalizes in the form of the following corollary.

\begin{corollary}\label{cor:sub}
If a graph $G$ contains an induced subgraph that is not sc-greedy, then $G$ is not sc-greedy.
\end{corollary}

\begin{proof}
Let $H$ be an induced subgraph of $G$ that is not sc-greedy.  The proof is by induction on $|V(G)|-|V(H)|$.  If $|V(G)|-|V(H)|=1$, then the result follows from Claim \ref{sub}.  Now assume the result holds for all induced subgraphs $H'$ for which $|V(G)|-|V(H')|<|V(G)|-|V(H)|$.  Specifically, it holds for $H\cup v$ for any $v\in N(H)$.  By Claim \ref{sub}, $H\cup v$ is not sc-greedy.  Since $|V(G)|-|V(H\cup v)|<|V(G)|-|V(H)|$, it follows by induction that $G$ is not sc-greedy.
\end{proof}

One might ask whether all triangulations are sc-greedy.
Corollary \ref{cor:sub} implies that not all triangulations are sc-greedy because it is possible to add three vertices of degree $4$ to $K_{2,3}$ to obtain a triangulation, see Figure \ref{fig:indsub}.  This graph is a triangulation with induced subgraph $K_{2,3}$, which is not sc-greedy.

\begin{figure}[h]
\begin{center}
\subfloat[][$K_{2,3}$]{\label{fig:K23}\includegraphics{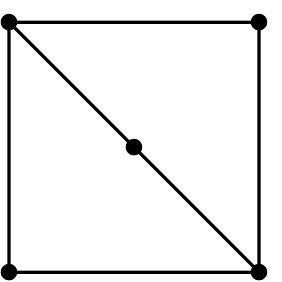}}
\qquad
\subfloat[][$K_{2,3}$ an induced subgraph of a triangulation]{\label{fig:K23plus}\includegraphics{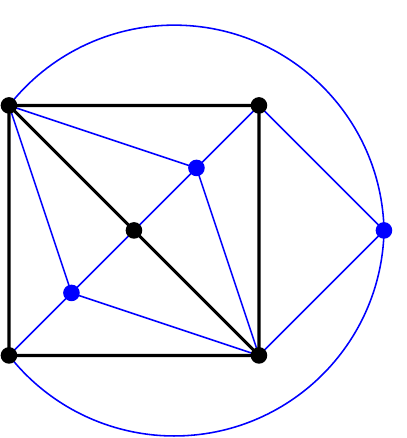}}
\end{center}
\caption{An example that illustrates not all triangulations are sc-greedy.}
\label{fig:indsub}
\end{figure}

We may also look at what can happen to the sum choice number of a graph upon addition of an edge.  One might predict that adding an edge to a graph would increase the sum choice number by at most one. However, this is not the case.  The following two facts can be observed:
\begin{enumerate}
\item There exist graphs that differ by an edge that have the same sum choice number.
\item There exist graphs that differ by an edge, but whose sum choice numbers differ by two.
\end{enumerate}
The first fact can be observed in Figure \ref{fig:edgesame} which displays two graphs that differ by an edge.  The graph in Figure \ref{fig:G1} is a $4$-cycle with a pendant path, hence sc-greedy, and the graph in Figure \ref{fig:G2} is $K_{2,3}$.  Both of these graphs have sum choice number $10$.
\begin{figure}[h]
\begin{center}
\subfloat[][$4$-cycle with a pendant path]{\label{fig:G1}\includegraphics{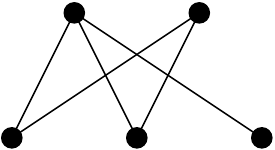}}
\qquad
\subfloat[][$K_{2,3}$]{\label{fig:G2}\includegraphics{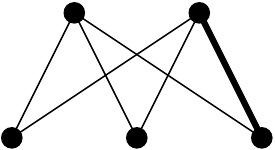}}
\end{center}
\caption{Two graphs that differ by an edge with the same sum choice number.}
\label{fig:edgesame}
\end{figure}
The second fact can be observed in Figure \ref{fig:edgeplus2} which displays two graphs that differ by an edge.  The graph in Figure \ref{fig:G3} is $K_{2,3}$, whose sum choice number is $10$, while the graph in Figure \ref{fig:G4} is a graph we call $G5.5$ whose sum choice number is $12$, see Section \ref{sec:five}.
\begin{figure}[h]
\begin{center}
\subfloat[][$K_{2,3}$]{\label{fig:G3}\includegraphics{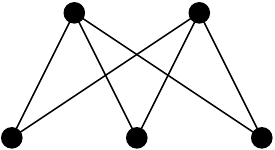}}
\qquad
\subfloat[][$G5.5$]{\label{fig:G4}\includegraphics{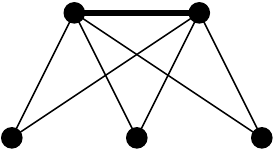}}
\end{center}
\caption{Two graphs that differ by an edge with sum choice numbers that differ by $2$.}
\label{fig:edgeplus2}
\end{figure}

\subsection{Edge subdivision and minors}
Some things can be said about edge subdivision and sc-greedy graphs.

\begin{claim}
There exist graphs that are sc-greedy for which it is possible to subdivide an edge and obtain a graph that is not sc-greedy.
\end{claim}

See Figure \ref{fig:subnosc} for an example.  An edge of $BW_3$, which is sc-greedy, can be subdivided to obtain $K_{2,3}$, which is not sc-greedy.

\begin{figure}[h]
\begin{center}
\subfloat[][$BW_3$]{\label{fig:BW3sub}\includegraphics{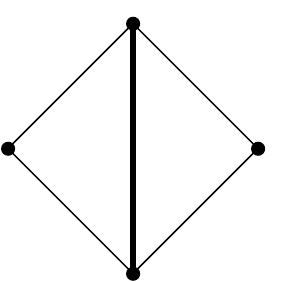}}
\qquad
\subfloat[][$K_{2,3}$]{\label{fig:K23suba}\includegraphics{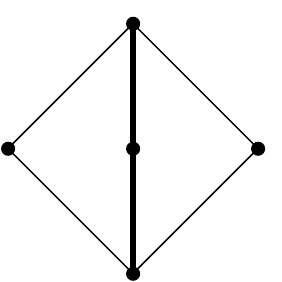}}
\end{center}
\caption{$K_{2,3}$ can be obtained by subdividing an edge of $BW_3$.}
\label{fig:subnosc}
\end{figure}

\begin{claim}
There exists graphs that are not sc-greedy for which it is possible to subdivide an edge and obtain a graph that is sc-greedy.
\end{claim}

See Figure \ref{fig:subsc} for an example.  An edge of $K_{2,3}$, which is not sc-greedy, can be subdivided to obtain $\Theta_{1,1,2}$, which is sc-greedy.

\begin{figure}[h]
\begin{center}
\subfloat[][$K_{2,3}$]{\label{fig:K23subb}\includegraphics{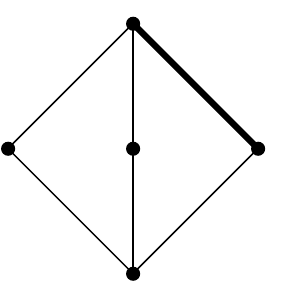}}
\qquad
\subfloat[][$\Theta_{1,1,2}$]{\label{fig:Theta112sub}\includegraphics{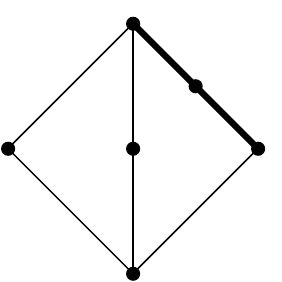}}
\end{center}
\caption{$\Theta_{1,1,2}$ can be obtained by subdividing an edge of $K_{2,3}$.}
\label{fig:subsc}
\end{figure}

The above two claims imply the following:
\begin{enumerate}
\item There exist graphs that are not sc-greedy with minors that are sc-greedy.
\item There exist graphs that are sc-greedy with minors that are not sc-greedy.
\end{enumerate}

\subsection{Minimally not sc-greedy graphs}\label{subsec:minnscg}
\begin{definition}
A graph $G$ is \textbf{minimally not sc-greedy} if for all $S\subset V(G)$, $G[S]$ is sc-greedy, but $G$ is not sc-greedy.
\end{definition}

Two examples of minimally not sc-greedy graphs are $K_{2,3}$ and $P_3\times P_3$.  Two more examples are the graphs  $G5.4$ (see Figure \ref{fig:G5_4}) and $G5.8$ (see Figure \ref{fig:G5_8}), which will be presented later.

\begin{question}\label{quest:minnscg}
If $G$ is minimally not sc-greedy, does it follow that $\chi_{SC}(G)$ is equal to $GB(G)-1$?
\end{question}

Note that for the examples of minimally not sc-greedy graphs mentioned above, the answer to this question is yes.\\

Assume $G$ is a minimally not sc-greedy graph for which $\chi_{SC}(G)=GB(G)-2$.  Then, there is a non-simple choice function $g$ for $G$ of size $GB(G)-2$ by Lemma \ref{greedy}.  Additionally, $\chi_{SC}(G-v)=GB(G-v)=GB(G)-\deg(v)-1$ for all $v\in V(G)$.  This implies the following:
\begin{eqnarray*}
  \size(g_{G-v}) &=& GB(G)-2-g(v) \\
   &\ge& \chi_{SC}(G-v) \\
   &=& GB(G)-\deg(v)-1 \\
  \Rightarrow \qquad GB(G)-2-g(v) &\ge& GB(G)-\deg(v)-1 \\
  \Rightarrow \qquad g(v) &\le& \deg(v)-1
\end{eqnarray*}
This must be true for all $v\in V(G)$, which implies that $\delta(G)\ge 3$.  However, the graphs $K_{2,3}$, $P_3\times P_3$, and $G5.4$ are all minimally not sc-greedy graphs of minimum degree $2$.  Thus, if $G$ is a minimally not sc-greedy graph for which $\delta(G)=2$, then $\chi_{SC}(G)=GB(G)-1$.
Note that $G5.8$ is a minimally not sc-greedy graph of minimum degree $3$.

Lemma \ref{greedy} implies that if a graph $G$ is minimally not sc-greedy, then (1) $G-v$ is sc-greedy for all $v\in V(G)$ and (2) there is a non-simple choice function $f$ for $G$ such that $\size(f)=GB(G)-1$.  Hence, the question remains if there is a non-simple choice function $f$ for $G$ such that $\size(f)<GB(G)-1$.  Does such a graph $G$ exist?  This would provide an answer to Question \ref{quest:minnscg}.

\subsection{Wheels and broken wheels}\label{sec:wbw}
Heinold \cite{Heinold06} explored the broken wheel $BW_k$ with respect to sum choosability.  In particular, he showed that $BW_{10}$ is not sc-greedy.  This provides some very useful information.  First, it implies that not all outerplanar graphs are sc-greedy.  It also implies that $BW_k$ will not be sc-greedy for all $k\ge10$ since such graphs will have $BW_{10}$ as an induced subgraph.  It was also shown that $BW_k$ is sc-greedy for all $k\le9$.  It may be observed here that $BW_{10}$ is minimally not sc-greedy.  Heinold also showed that there exist $k$ for which the gap between $GB(BW_k)$ and $\chi_{SC}(BW_k)$ is arbitrarily large.  While it is known which broken wheels are sc-greedy, it is not known what $\chi_{SC}(BW_k)$ is in general.  Heinold guessed that $\chi_{SC}(BW_k)=GB(BW_k)-\left\lfloor \frac{k+1}{11} \right\rfloor$.  See \cite{Heinold06} for more details.  His dissertation also establishes many techniques that could be used to determine $\chi_{SC}(BW_k)$.

As far as we know, the determination of $\chi_{SC}(BW_k)$ remains an open problem.  However, the results mentioned above can be used to obtain information about the sum choosability of wheels $W_k$.  Since $BW_{10}$ is not sc-greedy, it follows that $W_k$ is not sc-greedy for all $k\ge11$ since all of these graphs will have $BW_{10}$ as an induced subgraph.  Thus, in determining which wheels are sc-greedy, it remains to examine $W_k$ for $k\le10$.  For small values of $k$, the result follows quickly.  The graph $W_3$ is isomorphic to $K_4$, and thus is sc-greedy.  It was shown earlier that $W_4$ is not sc-greedy.  The classification would be completed by looking at $W_k$ for $k=5,\ldots,10$.

\begin{problem}
Determine whether or not $W_k$ is sc-greedy for $k=5,\ldots,10$.
\end{problem}

We summarize this information in the following theorem.
\begin{theorem}
Let $BW_k$ be a broken wheel and $W_k$ be a wheel.  The following is known:
\begin{enumerate}
\item if $k\le9$, then $BW_k$ is sc-greedy \cite{Heinold06},
\item if $k\ge10$, then $BW_k$ is not sc-greedy \cite{Heinold06},
\item if $k\le3$, then $W_k$ is sc-greedy, and
\item if $k=4$ or $k\ge11$, then $W_k$ is not sc-greedy.
\end{enumerate}
\end{theorem}

Some general observations may be made about what remains to be done.  First, for any $k$, removing an arbitrary vertex of $W_k$ will yield either a broken wheel or a cycle.  Thus, for $k=5,\ldots,10$, the graph $W_k-v$ is sc-greedy for all $v\in V(W_k)$.  To determine whether or not $W_k$ is sc-greedy, it must be determined whether or not there exists a non-simple choice function for $W_k$ of size $GB(W_k)-1=3k$.

\section{Trees of cycles}
In this section we show that trees of cycles are sc-greedy.  More specifically, we prove that paths of cycles are sc-greedy and the result for trees of cycles follows as a corollary.

Let $G$ be a path of $k$ cycles.
Recall that $G$ can be embedded in the plane so that the weak dual of $G$ is a path of length $k$.
For $i=2,\ldots,k$, let $L_i$ be the subgraph of $G$ induced by the vertices of $G_1,\ldots,G_{i-1}$, let $R_i$ be the subgraph of $G$ induced by the vertices of $G_i,\ldots,G_k$, and let $I_i=G[\{t_i,b_i\}]$.
Note that the greedy bound for $G$ is equal to $2\sum_{i=1}^k a_i-3(k-1)$.  

\begin{theorem}\label{thm:cyclepaths}
Paths of cycles are sc-greedy.
\end{theorem}

\begin{proof}
The proof is by induction on $k$.  The result clearly holds for $k=1$, as cycles are known to be sc-greedy.
Now assume paths of at most $m$ cycles are sc-greedy for all $m<k$.  It will be shown that the result holds for $k$.  

Assume $f$ is a minimal choice function for a graph $G$ which is a path of $k$ cycles.  The proof requires a case analysis on $f(t_i)+f(b_i)$ for $2\le i\le k$.  Observe here that $f(t_i)+f(b_i)\ge3$ as $I_i$ is isomorphic to $P_2$ and $\chi_{SC}(P_2)=3$.

Assume first that $f(t_i)=1$.
Now, in $G-t_i$ the vertex $b_i$ will be a cut-vertex that splits $G-t_i$ into two shorter paths of cycles, perhaps with pendant paths attached to them.  By induction, shorter paths of cycles are sc-greedy, hence attaching any pendant paths will also yield an sc-greedy graph.  Thus, $G-t_i$ is sc-greedy.\\

Furthermore, assume next there is a $j$ for which there exists $w_j\in V(G_j)$ such that $f(w_j)=1$, then in $G-w_j$ there is a cut-vertex that splits $G-w_j$ into two shorter paths of cycles, perhaps with pendant paths attached to them.  By induction, shorter paths of cycles are sc-greedy, hence attaching any pendant paths will also yield an sc-greedy graph.  Thus, $G-w_j$ is sc-greedy.\\

It follows that $G-v$ is sc-greedy for all $v\in V(G)$.  Thus, to show that $G$ is sc-greedy, it remains to show there does not exist a non-simple size function of size $GB-1$ for $G$.  So we assume that $f(v)>1$ for all $v\in V(G)$ for the remainder of the proof.\\

\noindent Case 1: $f(t_i)+f(b_i)\ge5$ for all $i=2,\ldots,k$.\\
Recall that $f(w_i)\ge2$ for all $w_i\ne t_i,b_i$.  So
\begin{eqnarray*}
  \size(f) &\ge& 2(a_1-2)+2\displaystyle\sum_{i=2}^{k-1}(a_i-4)+2(a_k-2)+5(k-1) \\
   &=& 2\displaystyle\sum_{i=1}^k a_i-3(k-1) \\
   &=& GB(G)
\end{eqnarray*}
and the result follows.\\

\noindent Case 2: There is a $j$ such that $f(t_j)=f(b_j)=2$.\\
Assume $\size(f)=GB(G)-1=2\displaystyle\sum_{i=1}^ka_i-3(k-1)-1$.

By induction,
\begin{eqnarray*}
  \chi_{SC}(L_j)  &  =  &  2\displaystyle\sum_{i=1}^{j-1}a_i-3(j-2), \\
  \chi_{SC}(R_j)  &  =  &  2\displaystyle\sum_{i=j}^ka_i-3(k-j),
\end{eqnarray*}
so it follows that
\begin{eqnarray*}
  \size(f_{L_j-\{t_j,b_j\}})  &  \ge  &  2\displaystyle\sum_{i=1}^{j-1}a_i-3(j-2)-4, \\
  \size(f_{R_j-\{t_j,b_j\}})  &  \ge  &  2\displaystyle\sum_{i=j}^ka_i-3(k-j)-4.
\end{eqnarray*}
In fact, the above inequalities must be equalities so that $\size(f_{L_j-\{t_j,b_j\}})+4+\size(f_{R_j-\{t_j,b_j\}})=\size(f)$ holds.
This allows for an $f$-assignment $L$ for which $G$ is not $L$-colorable to be defined as follows:

Let $g_1$ be a size function for $L_j-\{t_j,b_j\}$ such that
$$g_1(w) = \left\{
\begin{array}{lr}
f(w)-1  &  \text{if } w\sim t_j\text{ or }b_j\\
f(w)  &  \text{else}
\end{array}
\right.$$
and let $g_2$ be a size function for $R_j-\{t_j,b_j\}$ such that
$$g_2(w) = \left\{
\begin{array}{lr}
f(w)-1  &  \text{if } w\sim t_j\text{ or }b_j\\
f(w)  &  \text{else}
\end{array}
\right..$$
Thus $\size(g_1)<\chi_{SC}(L_j-\{t_j,b_j\})\le\size(f_{L_j-\{t_j,b_j\}})$ and $\size(g_2)<\chi_{SC}(R_j-\{t_j,b_j\})\le\size(f_{R_j-\{t_j,b_j\}})$, implying $L_j-\{t_j,b_j\}$ and $R_j-\{t_j,b_j\}$ are not $g_1$- and $g_2$-choosable, respectively.
There are $g_1$- and $g_2$-assignments $L_L$ and $L_R$, respectively, for which $L_j-\{t_j,b_j\}$ and $R_j-\{t_j,b_j\}$ cannot be list-colored.

Let $L$ be an $f$-assignment for $G$ defined as follows:  $L(b_j)=L(t_j)=\{c_1,c_2\}$ where $c_1$ and $c_2$ do not appear in any of the lists $L_L$ and $L_R$, $L=L_L$ on $L_j-\{t_j,b_j\}$ and $L=L_R$ on $R_j-\{t_j,b_j\}$, except append $c_1,\,c_2,\,c_2$ and $c_1$ to the lists of neighbors of $t_j$ and $b_j$ in $L_j$ and $R_j$, respectively.  If $c$ is a proper $L$-coloring of $G$, then either $c(t_j)=c_1$ and $c(b_j)=c_2$ or $c(t_j)=c_2$ and $c(b_j)=c_1$.

By the construction of $L$, this coloring will not provide a proper coloring of either $L_j$ or $R_j$, a contradiction.\\

\noindent Therefore, paths of cycles are sc-greedy.
\end{proof}

We now observe that Theorem \ref{thm:cyclepaths} extends to trees of cycles. 
Let $G$ be a tree of cycles. 
Recall that $G$ can be embedded in the plane so that the weak dual, call it $G'$, of $G$ is a tree on $k$ vertices.

Let $\mathcal{I}=\{\{i,j\}:V(G_i)\cap V(G_j)\neq\emptyset\}$.  Then $|\mathcal{I}|=|E(G')|$.  In particular, the number of pairs of cycles in $G$ that share vertices is equal to the number of edges in the weal dual of $G$.   This allows us to compute the greedy bound of $G$ as
$$GB(G)=2\sum_{i=1}^ka_i-3|\mathcal{I}|=2\sum_{i=1}^ka_i-3|E(G')|.$$
It then follows as a corollary to Theorem \ref{thm:cyclepaths} that trees of cycles are sc-greedy:

\begin{corollary}\label{cor:cycletrees}
Trees of cycles are sc-greedy. 
\end{corollary}

This result follows from Theorem \ref{thm:cyclepaths} and its proof because the argument and case analysis is applied to the intersection of $G_i$ with $G_j$ and the same properties will hold.  This leads us to ask the following questions:

\begin{question}
Are cycles of cycles sc-greedy?
\end{question}
\begin{question}
Are paths of cliques sc-greedy?
\end{question}
\begin{question}
If two sc-greedy graphs are joined by an edge, is the resulting graph sc-greedy?
\end{question}

\section{Acknowledgements}
The authors wish to thank Steve Butler for his contributions, including the work he did in Sage to help with the case analysis for graphs on five vertices.


\begin{thebibliography}{99}
\bibitem{BBBD06} Berliner, A., Bostelmann, U., Brualdi, R. A., Deatt, L.: Sum list coloring graphs.  Graphs Combin. 22 no. 2, 173--183 (2006)

\bibitem{ERT} Erd\H{o}s, P., Rubin, A. L., Taylor, H.: Choosability in graphs.  Proceeding of the West Coast Conference on Combinatorics, Graph Theory and Computing, Congress Numer. 27, Utiliitas Math., 125--157 (1980) 

\bibitem{FK09} F\"{u}redi, Z., Kantor, I.: List coloring with distinct list sizes, the case of complete bipartite graphs.  European Conference on Combinatorics, Graph Theory and Applications, Electron. Notes Discret. Math. 34, 323--327 (2009)


\bibitem{Heinold07} Heinold, B.: A survey of sum list coloring.  Graph Theory Notes N. Y. 52, 38--44 (2007)

\bibitem{Heinold12} Heinold, B.: The sum choice number of $P_3\Box P_n$.  Discret. Appl. Math. 160, 1126--1136 (2012)

\bibitem{Heinold09} Heinold, B.: Sum choice numbers of some graphs.  Discret. Math. 309 no. 8, 2166-2173 (2009)

\bibitem{Heinold06} Heinold, B.: Sum list coloring and choosability. Ph.D. Thesis, Lehigh University, 88 pp. (2006)

\bibitem{Isaak02} Isaak, G.: Sum list coloring $2\times n$ arrays.  Electron. J. Combin.  9 no. 1 Note 8, 7 pp. (2002)

\bibitem{Isaak04} Isaak, G.: Sum list coloring block graphs.  Graphs Combin. 20 no. 4, 499--506 (2004)

\bibitem{Lastrina12} Lastrina, M.: List-coloring and sum-list-coloring problems on graphs.  Ph.D. Thesis, Iowa State University, 127 pp. (2012)
\end{thebibliography}
\end{document}